\documentclass[11pt]{amsart}
\usepackage{tikz-cd}

\usepackage[backref]{hyperref}
\usepackage{color}
\definecolor{mylinkcolor}{rgb}{0.8,0,0}
\definecolor{myurlcolor}{rgb}{0,0,0.8}
\definecolor{mycitecolor}{rgb}{0,0,0.8}
\hypersetup{colorlinks=true,urlcolor=myurlcolor,citecolor=mycitecolor,linkcolor=mylinkcolor,linktoc=page,breaklinks=true}

\makeatletter
\@namedef{subjclassname@2010}{%
  \textup{2010} Mathematics Subject Classification}
\makeatother

\newtheorem{corollary}{Corollary}
\newtheorem{lemma}[corollary]{Lemma}
\newtheorem{proposition}[corollary]{Proposition}
\newtheorem{theorem}[corollary]{Theorem}

\theoremstyle{definition}
\newtheorem{remark}[corollary]{Remark}
\newtheorem{example}[corollary]{Example}

\numberwithin{corollary}{section}

\newcommand{\Q}{\mathbb Q}
\newcommand{\Z}{\mathbb Z}
\newcommand{\F}{\mathbb F}
\newcommand{\PP}{\mathbb P}

\newcommand{\Gal}{\operatorname{Gal}}
\newcommand{\id}{\operatorname{id}}
\newcommand{\Pic}{\operatorname{Pic}}
\newcommand{\red}{\operatorname{red}}
\newcommand{\SL}{\operatorname{SL}}
\newcommand{\Sym}{\operatorname{Sym}}
\newcommand{\tor}{\mathrm{tor}}

\newcommand{\smallmat}[4]{\left(\begin{smallmatrix}#1 & #2 \\ #3 & #4\end{smallmatrix}\right)}

\newcommand{\diam}[1]{\langle#1\rangle}
\newcommand{\isom}{\buildrel\sim\over\longrightarrow}

\begin{document}

\title{Fields of definition of elliptic curves with prescribed torsion}

\author[P.\thinspace J. Bruin]{Peter Bruin}
\address{Mathematisch Instituut\\ Universiteit Leiden\\ Postbus 9512\\ 2300 RA \ Leiden\\ Netherlands}
\email{P.J.Bruin@math.leidenuniv.nl}
\thanks{The first author was supported by the Netherlands Organisation for Scientific Research (NWO) through Veni grant 639.031.346}
\author[F. Najman]{Filip Najman}
\address{Department of Mathematics\\ University of Zagreb\\ Bijeni\v cka cesta 30\\ 10000 Zagreb\\ Croatia}
\email{fnajman@math.hr}
\thanks{The second author gratefully acknowledges support from the QuantiXLie Center of Excellence.}

\begin{abstract}
	We prove that all elliptic curves over quadratic fields with a subgroup isomorphic to $C_{16}$, as well as all elliptic curves over cubic fields with a subgroup isomorphic to $C_2\times C_{14}$, are base changes of elliptic curves defined over $\Q$. We obtain these results by studying geometric properties of modular curves and maps between modular curves, and then obtaining a modular description of these curves and maps.
\end{abstract}



\maketitle

\section{Introduction}

By the Mordell--Weil theorem, the Abelian group $E(K)$ of $K$-rational points on an elliptic curve~$E$ over a number field~$K$ is finitely generated. This group can therefore be decomposed as $E(K)\simeq E(K)_\tor\oplus \Z^r$, where $r$ is the rank of~$E$ over~$K$.

Let $\Phi(d)$ denote the set of isomorphism classes of finite groups $G$ with the property that there exists an elliptic curve $E$ over a number field $K$ of degree~$d$ such that $E(K)_\tor\simeq G$. In this paper we will show that for $d=2$ and $d=3$ and for certain groups $G\in \Phi(d)$, if $E(K)_\tor\simeq G$, it turns out that $E$ is a base change of an elliptic curve over $\Q$.

The first example of a result where the torsion of an elliptic curve over a number field of given degree yields information about its field of definition can be found in \cite{bbdn}. There it was shown that if an elliptic curve over a quadratic field~$K$ has a point of order $13$ or~$18$, then $K$ is a real quadratic field. In other words, there are no elliptic curves over imaginary quadratic fields with a point of order $13$ or~$18$. Another result in the same paper shows that if an elliptic curve over a quartic field~$K$ has a point of order~$22$, then $K$ has a quadratic subfield over which the modular curve $Y_1(11)$ has points; note that ``most'' quartic fields do not have quadratic subfields. In \cite{dkm}, it is proved that if an elliptic curve over a quartic field~$K$ has a point of order~$17$ and $L$ is the normal closure of $K$ over~$\Q$, then $\Gal(L/\Q)$ is isomorphic to $D_4$ or~$S_4$.

The goal of this paper is to prove the following two theorems.

\begin{theorem}
\label{theorem:quadratic}
Every elliptic curve over a quadratic field with a subgroup isomorphic to $C_{16}$ is a base change of an elliptic curve over\/~$\Q$ with a subgroup isomorphic to~$C_8$.
\end{theorem}

\begin{theorem}
\label{theorem:cubic}
If $E$ is an elliptic curve over a cubic field $K$ with a subgroup isomorphic to $C_2\times C_{14}$, then $K$ is normal over\/~$\Q$ and $E$ is a base change of an elliptic curve over\/~$\Q$.
\end{theorem}

We found examples of elliptic curves that are not base changes of elliptic curves over $\Q$ with all possible torsion groups over quadratic fields apart from $C_{16}$, and with all possible torsion groups over cubic fields apart from $C_2\times C_{14}$ and $C_{21}$. The case of $C_{21}$ is somewhat special: there is a unique curve over a cubic field with this torsion group, and the curve is a base change of an elliptic curve (with Cremona label 162b1) over $\Q$.  This curve was found by the second author \cite{naj} and was proved to be the only such curve in yet unpublished work of Derickx, Etropolski, Morrow and Zureick-Brown.

In Section 2 we prove Theorem \ref{theorem:quadratic}; in Section 3 we prove the more difficult Theorem \ref{theorem:cubic}. We deduce these results from geometric properties of modular curves and maps between modular curves, combinined with the modular description of these curves and maps.
For a congruence subgroup $\Gamma\subseteq\SL_2(\Z)$, let $X_\Gamma$ be the corresponding modular curve and let $Y_\Gamma$ be the complement of the cusps in $X_\Gamma$.

The idea of the proof of Theorem \ref{theorem:quadratic} is as follows. There exists a congruence subgroup $\Gamma'\subset\SL_2(\Z)$, containing $\Gamma_1(16)$ as a subgroup of index~$2$, such that all the degree~$2$ points on $Y_1(16)$ map to $\Q$-rational points of~$Y_{\Gamma'}$ under the natural morphism $X_1(16) \rightarrow X_{\Gamma'}$. The modular descriptions of $X_1(16)$ and~$X_{\Gamma'}$ then allow us to conclude that the points of degree~$2$ on $Y_1(16)$ in fact parameterize elliptic curves defined over~$\Q$.

Moving on to the setting of Theorem~\ref{theorem:cubic}, let $\Gamma=\Gamma_1(2,14)$. The proof of Theorem~\ref{theorem:cubic} follows the same lines, with the difference that in this case there will be two maps $q_H\colon X_\Gamma\rightarrow X_\Gamma/H$ and $q_{H'}\colon X_\Gamma\rightarrow X_\Gamma/H'$ of degree~3, where the quotient curves $X_\Gamma/H$ and $X_\Gamma/H'$ of genus~0 are constructed in Section~\ref{subec:3.1}. We prove that all cubic points of $Y_\Gamma$ are inverse images of rational points of $Y_\Gamma/H$ or $Y_\Gamma/H'$ under the maps $q_H$ and $q_{H'}$. Using an explicit equation for~$X_\Gamma$, we compute the group of $\Q$-points of the Jacobian of~$X_\Gamma$ (Proposition~\ref{prop_j2_14}) and describe the set of effective divisors of degree~$3$ on~$X_\Gamma$ (Proposition~\ref{prop:1}).  It is then not hard to deduce that all cubic points on $Y(\Gamma)$ arise from $\Q$-rational points on $Y_\Gamma/H$ or $Y_\Gamma/H'$.

\section{Elliptic curves with $16$-torsion over quadratic fields}

In this section we will prove Theorem~\ref{theorem:quadratic}.
We write
$$
\Gamma'=\left\{\begin{pmatrix} a & b \\ c & d \end{pmatrix}
\in \SL_2(\Z)\biggm|
\begin{aligned}
a\equiv d&\equiv 1 \pmod 8,\\
\noalign{\vskip-\jot}
c&\equiv 0 \pmod{16}
\end{aligned}
\right\}.
$$
The curves $X_1(16)$ and~$X_{\Gamma'}$ have genus $2$ and~$0$, respectively, and the map
$$
\pi\colon X_1(16)\rightarrow X_{\Gamma'}
$$
of degree~$2$ is a quotient map for the diamond automorphism $\diam{9}$ on $X_1(16)$.
It was already shown in \cite{bbdn} that all quadratic points on $Y_1(16)$ are inverse images under $\pi$ of $\Q$-rational points of $Y_{\Gamma'}$.

\begin{proof}[Proof of Theorem \ref{theorem:quadratic}]
Consider a point of $Y_1(16)(K)$ corresponding to a pair $(E,P)$, where $E$ is an elliptic curve over a quadratic field~$K$ and $P\in E(K)$ is a point of order~16. Let $\sigma$ be the generator of $\Gal(K/\Q)$.
Using the fact that the hyperelliptic involution on $X_1(16)$ is the diamond automorphism $\diam{9}$, it was proved in \cite[\S\,4.5]{bbdn} that there exists an isomorphism
$$
\mu\colon E^\sigma \isom E
$$
satisfying
$$
\mu \circ \mu^\sigma = \id
\quad\text{and}\quad
\mu(P^\sigma) = 9P.
$$
(This $\mu$ differs from the one in \cite{bbdn} by a sign.)  It follows, although this was not explicitly stated in~\cite{bbdn}, that $E$ can be descended to~$\Q$.  The isomorphism $\mu$ maps $(2P)^\sigma$ to $2(9P)=18P=2P$. Therefore not only $E$, but also the point $2P$ of order~$8$ is defined over~$\Q$.
\end{proof}

The above argument can be made explicit as follows.
The modular curve $X_1(16)$ admits the equation
$$
X_1(16)\colon v^2 - (u^3 + u^2 - u + 1)v + u^2 = 0.
$$
From \cite{bbdn}, it follows that all quadratic points $(u,v)$ on $Y_1(16)$ satisfy $u\in \Q$.
One can write down, in terms of the coordinates $(u, v)$, equations for the universal elliptic curve~$E$ and for the universal point~$P$ of order~$16$ on~$E$.
(The resulting equations are the same as those obtained by Rabarison \cite[Section 4.4]{rab}, up to a change of variables in the equation for~$X_1(16)$.)
One can then descend the pair $(E, 2P)$ to~$\Q$ by writing $E$ in Tate normal form with respect to the point $2P$.  This gives the Weierstrass equation
$$
E\colon y^2 + axy + by = x^3 + bx^2 \text{ with } 2P = (0, 0),
$$
where
$$
a = 1 - \frac{u^2(u - 1)(u + 1)}{u^2 + 1}
\quad\text{and}\quad
b = \frac{-u^2(u - 1)(u + 1)}{(u^2 + 1)^2}.
$$
Since these expressions do not contain $v$, we obtain a Weierstrass equation for~$E$ with coefficients in $\Q$.

\section{Elliptic curves with $(2,14)$-torsion over cubic fields}

Next, we take $\Gamma=\Gamma_1(2,14)=\Gamma(2)\cap\Gamma_1(7)$.
We will study $\Gamma$ and the corresponding modular curve $X_\Gamma$ using several auxiliary congruence subgroups.
Let $\Gamma_*(2)$ be the unique subgroup of $\SL_2(\Z)$ that contains $\Gamma(2)$ and such that $(\Gamma_*(2) : \Gamma(2)) = 3$; more precisely,
$$
\Gamma_*(2)=\left\{ \Gamma \in \SL_2(\Z) \biggm|
\gamma \equiv \begin{pmatrix} 1 & 0 \\ 0 & 1 \end{pmatrix},
\begin{pmatrix} 1 & 1 \\ 1 & 0 \end{pmatrix},
\begin{pmatrix} 0 & 1 \\ 1 & 1 \end{pmatrix}
\pmod 2 \right \}.
$$
We also define
$$
\Gamma_*(7)=\left\{ \begin{pmatrix} a & b \\ c & b \end{pmatrix}
\in \SL_2(\Z)\biggm|
\begin{aligned}
a,d&\equiv 1,2,4 \pmod 7,\\
\noalign{\vskip-\jot}
c&\equiv 0 \pmod 7
\end{aligned}
\right\}.
$$
We note that $\Gamma_*(7)$ does not contain the matrix $\smallmat{-1}{0}{0}{-1}$. It does have two conjugacy classes of elliptic elements of order~$3$, corresponding to two specific $\Gamma_*(7)$-structures on an elliptic curve with $j$-invariant 0.

The groups $\Gamma_*(2)$ and $\Gamma_*(7)$ contain $\Gamma(2)$ and $\Gamma_1(7)$, respectively, as normal subgroups of index~$3$.  We define $A_3$ and $C_3$ as the respective quotients $\Gamma_*(2)/\Gamma(2)$ and $\Gamma_*(7)/\Gamma_1(7)$, and we make the identifications
$$
\begin{aligned}
A_3&=(\Gamma_*(2) \cap \Gamma_1(7))/\Gamma,\\
C_3&=(\Gamma(2) \cap \Gamma_*(7))/\Gamma,\\
A_3\times C_3&=(\Gamma_*(2) \cap \Gamma_*(7))/\Gamma.
\end{aligned}
$$
The group $A_3\times C_3$ has four subgroups of order~$3$; besides $A_3$ and $C_3$, there are two further subgroups $H$ and~$H'$.

\subsection{Geometric properties of modular curves}
\label{subec:3.1}

The modular curve $X_\Gamma$ equals $X_1(2,14)$. Furthermore, the modular curve $X_{\Gamma_*(7)}$ is just $X_0(7)$, but we will denote it by $X_*(7)$ in view of the fact that it is defined using $\Gamma_*(7)$ instead of $\Gamma_0(7)$, which is essential to our method. We will also need the modular curves $X_1(7)$ and $X_1(14)$. The curves $X_*(7)$ and $X_1(7)$ have genus~$0$. The curve $X_1(14)$ has genus~$1$, and is isomorphic to the elliptic curve over~$\Q$ with Cremona label 14a4.

The group $\Gamma_*(7)/\Gamma$ acts on $X_\Gamma$.
The action of the various subgroups of interest gives rise to the following diagram, where the numbers next to the arrows indicate the degrees:
$$
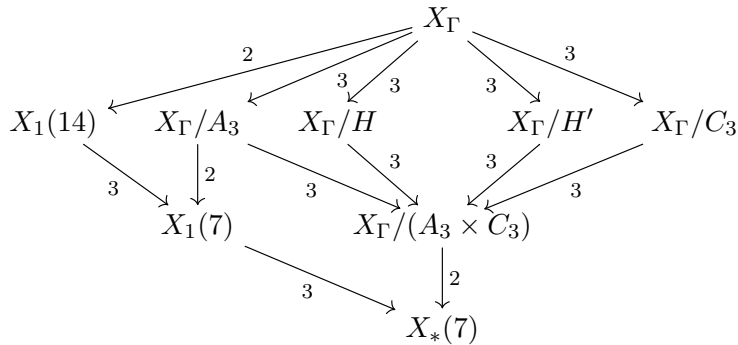
\begin{tikzcd}[column sep=small]
\null & & & X_\Gamma \arrow{dlll}[swap]{2} \arrow{dll}{3} \arrow{dl}{3} \arrow{dr}[swap]{3} \arrow{drr}{3} \\
X_1(14) \arrow{dr}[swap]{3} & X_\Gamma/A_3 \arrow{d}{2} \arrow{drr}[swap]{3} & X_\Gamma/H\kern-1em \arrow{dr}{3} & & \kern-1em X_\Gamma/H' \arrow{dl}[swap]{3} & X_\Gamma/C_3 \arrow{dll}{3} \\
\null & X_1(7) \arrow{drr}[swap]{3} & & \kern-2em X_\Gamma/(A_3\times C_3)\kern-2em \arrow{d}{2} \\
\null & & & X_*(7)
\end{tikzcd}
$$

The \emph{index} of a cusp on a modular curve $X$ is the order of vanishing of the discriminant modular form, or equivalently the ramification index of the canonical map $X \rightarrow X(1)$, at this cusp.

\begin{lemma}
The curve $X_\Gamma$ has genus~$4$. It has $18$ cusps: $9$ of index~$2$ and $9$ of index~$14$.
\label{lemma_cusp}
\end{lemma}

\begin{proof}
The map $X_\Gamma\to X_1(14)$ of degree~$2$ is unramified over the open subset $Y_1(14)$; this follows for example from the fact that there is a universal elliptic curve over $Y_1(14)$ and that its $2$-torsion is \'etale. As for the cusps, for each $d \in \{1, 2, 7, 14\}$ there are three cusps of index $d$ on $X_1(14)$, and the above covering is ramified exactly above the six cusps of index $1$ or~$7$ on $X_1(14)$. The Hurwitz formula gives
$$
2g(X_\Gamma) - 2 = 2(2g(X_1(14)) - 2) + 6.
$$
Both statements now follow easily.
\end{proof}

\begin{lemma}
\label{lemma:free}
The groups $A_3$ and $C_3$ act freely on $X_\Gamma$.
\end{lemma}

\begin{proof}
The action of the group $C_3$ on~$X_\Gamma$ descends to an action on $X_1(14)$ via the group of diamond automorphisms $\{\diam{1},\diam{9},\diam{11}\}$. Under any identification of $X_1(14)$ with an elliptic curve, the automorphisms $\diam{9}$ and~$\diam{11}$ act as translations by $3$-torsion points and hence have no fixed points.
It follows that $C_3$ acts freely on $X_1(14)$, and hence also on~$X_\Gamma$.

The group $A_3$ acts freely on $Y_\Gamma$ because $Y_1(7)$ is a fine moduli space. The cusps also have trivial stabilizer; this follows from the fact that the indices of all cusps are coprime to the order of~$A_3$.
\end{proof}

\begin{corollary}
The quotient maps
\begin{align*}
X_\Gamma&\rightarrow X_\Gamma/A_3,\\
X_\Gamma&\rightarrow X_\Gamma/C_3
\end{align*}
are unramified. Each of the curves $X_\Gamma/A_3$ and $X_\Gamma/C_3$ has $6$ cusps: $3$ of index $2$ and $3$ of index $14$. Both curves have genus~$2$.
\end{corollary}

\begin{proof}
The first two statements are immediate from Lemma~\ref{lemma:free}; the last one follows from the Hurwitz formula.
\end{proof}

\begin{lemma}
\begin{enumerate}
\item The curve $X_\Gamma/(A_3 \times C_3)$ has genus~$0$. It has two cusps: one of index $2$ and one of index $14$.
\label{lemma:quotient1}
\item The map $X_\Gamma\rightarrow X_\Gamma/(A_3\times C_3)$ has ramification index~$3$ at $12$ points of $X_\Gamma$ (lying above $4$ points of $X_\Gamma/(A_3\times C_3)$) and is unramified everywhere else.
\label{lemma:quotient2}
\end{enumerate}
\label{lemma:quotient}
\end{lemma}

\begin{proof}
The map $X_\Gamma/A_3 \rightarrow X_1(7)$ is unramified over $Y_1(7)$ by Lemma~\ref{lemma:free}, so the map $X_\Gamma/(A_3\times C_3) \rightarrow X_*(7)$ is unramified outside the cusps. Since $X_*(7)$ has genus~$0$, the Hurwitz formula implies that this last map is ramified above the two cusps of $X_*(7)$ and that $X_\Gamma/(A_3 \times C_3)$ has genus~$0$.  This proves (\ref{lemma:quotient1}).

By the Hurwitz formula and the fact that the map
$X_\Gamma/A_3\rightarrow X_\Gamma/(A_3\times C_3)$
is cyclic of degree~$3$, this map is totally ramified at 4 points.  The claim (\ref{lemma:quotient2}) now follows from the fact that the map $X_\Gamma\rightarrow X_\Gamma/A_3$ is unramified (the same argument works for $C_3$).
\end{proof}

\begin{lemma}
The curves $X_\Gamma/H$ and $X_\Gamma/H'$ have genus~$0$.
\end{lemma}

\begin{proof}
Let $P$ be one of the 12 ramification points of the map $X_\Gamma\to X_\Gamma/(A_3\times C_3)$. Then the stabilizer $G_P$ of~$P$ in $A_3 \times C_3$ is of order~$3$ and different from $A_3$ and $C_3$ since the latter two groups act freely on $X_\Gamma$ by Lemma~\ref{lemma:free}. Therefore $G_P$ is either $H$ or~$H'$. Let $n_H$ be the number of points $P \in X_\Gamma$ with stabilizer~$H$, and similarly for~$n_{H'}$. The Hurwitz formula gives
$$
2g(X_\Gamma) - 2 = 3(g(X_\Gamma/H) - 2) + 2n_H,
$$
$$
2g(X_\Gamma) - 2 = 3(g(X_\Gamma/H') - 2) + 2n_{H'}.
$$
Adding the two equations and using $g(X_\Gamma) = 4$ and $n_H + n_{H'} = 12$, we get $g(X_\Gamma/H) + g(X_\Gamma/H') = 0$, which implies
the claim.
\end{proof}

We conclude that the curve $X_\Gamma$ admits two maps of degree~$3$ to a curve of genus~$0$, namely the quotient maps
$$
q_H\colon X_\Gamma \rightarrow X_\Gamma/H,\quad
q_{H'}\colon X_\Gamma \rightarrow X_\Gamma/{H'}.
$$
By construction, both are cyclic with Galois groups $H$ and~$H'$, respectively.
Pull-back of divisors along the two maps $q_H$ and $q_{H'}$ gives rise to two lines $L$ and $L'$ (copies of~$\PP^1_\Q$) inside $\Sym^3 X_\Gamma$.
Both maps are ramified at exactly 6 points, and the two sets of 6 points are disjoint because of Lemma \ref{lemma:quotient}(\ref{lemma:quotient2}). This implies that $X_\Gamma$ embeds as a smooth curve of bidegree $(3, 3)$ in $X_\Gamma/H \times X_\Gamma/H' \simeq \PP^1_\Q \times \PP^1_\Q$, and that $L$ and~$L'$ are disjoint.
Furthermore, because a curve of genus~$4$ admits at most two linear systems of degree~$3$ and dimension~$1$ (see \cite[IV, Example~5.2.2]{har}), every non-constant map $X_\Gamma \rightarrow \PP^1_\Q$ of degree~$3$ can be identified with either $q_H$ or~$q_{H'}$
via an isomorphism of $\PP^1_\Q$ with $X_\Gamma/H$ or $X_\Gamma/H'$, respectively.

We fix one rational cusp, say $O=(0,0)$, and we consider the Jacobian $J_\Gamma$ of $X_\Gamma$ and the (non-dominant) rational map
\begin{equation}
\begin{aligned}
\phi\colon\Sym^3 X_\Gamma &\rightarrow J_\Gamma \\
D &\mapsto [D-3O]
\end{aligned}
\label{eq:phi}
\end{equation}

\begin{lemma}
  The map $\phi$ contracts the lines $L$ and $L'$ and is injective outside $L\cup L'$.
\end{lemma}

\begin{proof}
  Consider two distinct points of $\Sym^3 X_\Gamma$ corresponding to effective divisors $D$, $D'$ of degree~$3$ on $X_\Gamma$.  Then $\phi(D)=\phi(D')$ if and only if $D$ and~$D'$ are linearly equivalent.  In this case, there exists a rational function $f$ on~$X_\Gamma$ with divisor $D-D'$.  Such an $f$ gives a map of degree at most $3$ to $\PP^1_\Q$; this can be identified with either $q_H$ or $q_{H'}$, since $X_\Gamma$ is not hyperelliptic.  This implies that $\phi(D)=\phi(D')$ if and only if either both $D$ and~$D'$ are pull-backs of points under~$q_H\colon X_\Gamma\to X_\Gamma/H$, or both are pull-backs of points under $q_{H'}\colon X_\Gamma\to X_\Gamma/H'$.
\end{proof}

\begin{proposition}
The modular curve $X_\Gamma$ is isomorphic to the smooth projective curve of bidegree $(3,3)$ in $\PP^1_\Q\times\PP^1_\Q$ given by the equation
\begin{equation}
X_\Gamma\colon (u^3 + u^2 - 2u - 1)v(v + 1)
+ (v^3 + v^2 - 2v - 1)u(u + 1) = 0.
\label{eq:X_Gamma}
\end{equation}
The $(u,v)$-coordinates of the $9$ rational cusps are
$$
\displaylines{
(0,0),\quad
(0,-1),\quad
(0,\infty),\cr
(-1,0),\quad
(-1,-1),\quad
(-1,\infty),\cr
(\infty,0),\quad
(\infty,-1),\quad
(\infty,\infty).}
$$
The $9$ cusps with field of definition $\Q(\zeta_7)^+$ are defined by
$$
u^3 + u^2 - 2u - 1 = v^3 + v^2 - 2v - 1 = 0.
$$
\end{proposition}

Our initial proof of the above proposition proceeded by viewing $X_\Gamma$ as the $S_3$-cover of the curve $X_1(7)$ corresponding to the moduli problem of labelling the three points of order~2 by the set $\{0,1,2\}$.  As this proof involves rather long calculations, we do not give it here, but refer to the independent derivation of the above equation for $X_\Gamma$ by Derickx and Sutherland \cite[\S\,3.1]{ds}.

\subsection{Proof of the main result}

We first determine the structure of $J_\Gamma(\Q)$.
As $X_\Gamma$ is a non-hyperelliptic genus 4 curve, note that $J_\Gamma(\Q)$ cannot be computed directly in any current computer algebra system.

\begin{proposition}
\label{prop_j2_14}
The group $J_\Gamma(\Q)$ is generated by differences of rational cusps and is isomorphic to $C_2\times C_2\times C_6 \times C_{18}$.
\end{proposition}

\begin{proof}
The modular Abelian variety $J_\Gamma$ over~$\Q$ decomposes up to isogeny as $J_\Gamma\sim E\times E \times B$, where $E$ is an elliptic curve and $B$ is an Abelian surface.
A computation with newforms in either Magma or Sage \cite{sage} shows that the $L$-functions of $E$ and~$B$ do not vanish at~$1$.  By results of Kato \cite{kato}, the Birch--Swinnerton-Dyer conjecture is true for modular Abelian varieties.  We conclude that $J_\Gamma(\Q)$ has rank~$0$.

Let $\red_3$ denote the reduction map $J_\Gamma(\Q)\to J_\Gamma(\F_3)$.  Then $\red_3$ is injective.  One computes the numerator of the zeta function of $X_\Gamma$ over $\F_3$ to be $1 + 5x + 12x^{2} + 17x^{3} + 22x^{4} + 51x^{5} + 108x^{6} + 135x^{7} + 81x^{8}$.  Looking at the coefficient of~$x$, we obtain $\#X_\Gamma(\F_3) = 1 + 5 + 3 = 9$; substituting $x=1$, we obtain $\#J_\Gamma(\F_3)=432 = 2^4 \cdot 3^3$.  We deduce that $\#J_\Gamma(\Q)$ divides $432$.

Let $A$ be the subgroup of $J_\Gamma(\Q)$ generated by all differences of two rational cusps.
Then $A$ can be written as $A_2\times A_3$, where $A_2$ and $A_3$ are the 2-primary and 3-primary subgroups of $A$, respectively, and it suffices to compute $A_2$ and $A_3$.  The above bound on $\#J_\Gamma(\Q)$ implies that $\#A_2$ divides $2^4$ and $\#A_3$ divides $3^3$.
We claim that there are isomorphisms
$$
(\Z/2\Z)^4\isom A_2,\quad
\Z/3\Z\times\Z/9\Z\isom A_3.
$$
We will prove this by computing the images of $A_2$ and~$A_3$ under $\red_3$.

To compute in $J_\Gamma(\F_3)$, we use Khuri-Makdisi's algorithmic framework for computing in Picard groups of projective curves \cite{km1,km2}.  For a curve $X$ over a field $k$, with Jacobian $J$, this gives us a way to represent elements of $J(k)\simeq\Pic^0 X$ and algorithms to perform the following operations:
\begin{itemize}
\item given two points $P,Q\in X(k)$, compute the divisor class $[P-Q]\in J(k)$;
\item given two elements $x,y\in J(k)$, compute $-x-y$ (which also allows us to perform addition and negation);
\item given an element $x\in J(k)$, test whether $x$ is the zero element (which also allows us to test whether two elements are equal);
\item given elements $x\in J(k)$ and $O\in X(k)$, compute the least $r\ge0$ such that $x$ is of the form $[D-rO]$ for some effective divisor $D$ of degree~$r$.
\end{itemize}
We used an unpublished implementation of Khuri-Makdisi's algorithms over finite fields by the first named author in PARI/GP \cite{pari}.  For this we need to determine the space of global sections of a line bundle of sufficiently high degree.  Starting from the equation \eqref{eq:X_Gamma} and using the line bundle $\mathcal{O}_{X_\Gamma}(2((0,\infty) + (-1,\infty) + (\infty,0) + (\infty,-1) + (\infty,\infty)))$ of degree~$10$, we obtain the basis $(1, u, v, uv, u^2, v^2, uv(u + v))$ for the space of global sections.

For every point $P\in X_\Gamma(\F_3)$, we consider the corresponding point $[P-(0,0)]\in J_\Gamma(\F_3)$.
We define the following elements of $J_\Gamma(\F_3)$:
$$
\begin{aligned}
  x_1 &= 9[(0,-1)-(0,0)], \\
  x_2 &= 9[(0,\infty)-(0,0)], \\
  x_3 &= 9[(-1,0)-(0,0)], \\
  x_4 &= 9[(-1,-1)-(0,0)], \\
\end{aligned}
\qquad
\begin{aligned}
  y_1 &= 2[(-1,0)-(0,0)], \\
  y_2 &= 2[(-1,-1)-(0,0)]. \\ \\ \\
\end{aligned}
$$
Then the points $x_i$ have order~$2$, the point $y_1$ has order~$9$, and the point~$y_2$ has order~$3$.
We consider the group homomorphisms
$$
\begin{aligned}
  \lambda_2\colon (\Z/2\Z)^4 &\longrightarrow \red_3(A_2)\cr
  (a_1,a_2,a_3,a_4) &\longmapsto \sum_{i=1}^4 a_i x_i.
\end{aligned}
$$
and
$$
\begin{aligned}
  \lambda_3\colon \Z/9\Z\times\Z/3\Z &\longrightarrow \red_3(A_3)\cr
  (b_1,b_2) &\longmapsto b_1 y_1 + b_2 y_2.
\end{aligned}
$$
These fit in the following commutative diagrams:
$$
\begin{tikzcd}
(\Z/2\Z)^4 \arrow{r} \arrow{dr}[swap]{\lambda_2} & A_2 \arrow{d}{\red_3} \\
& \red_3(A_2)
\end{tikzcd}
\qquad
\begin{tikzcd}
\Z/9\Z\times\Z/3\Z \arrow{r} \arrow{dr}[swap]{\lambda_3} & A_3 \arrow{d}{\red_3} \\
& \red_3(A_3)
\end{tikzcd}
$$
where the vertical maps $A_2\to\red_3(A_2)$ and $A_3\to\red_3(A_3)$ are isomorphisms.
We show that $\lambda_2$ is injective by evaluating $\lambda_2$ on each element of $(\Z/2\Z)^4$ and testing whether the result is zero.  In a similar way, we show that $\lambda_3$ is injective.  Comparing orders, we see that $\lambda_2$ and $\lambda_3$ are isomorphisms.  Therefore both $A$ and $\red_3(A)$ are isomorphic to $C_2\times C_2\times C_6 \times C_{18}$, and in particular have order~$432$.  Finally, we deduce $J_\Gamma(\F_3)=\red_3(A)$ and $J_\Gamma(\Q) = A$.
\end{proof}

We now determine the image of the set of divisors of degree~$3$ under the map~$\phi$ defined by~\eqref{eq:phi}.

\begin{proposition}
\label{prop:1}
The image of $(\Sym^3 X_\Gamma)(\Q)$ under~$\phi$ equals the set of points in $J_\Gamma(\Q)$ represented by effective divisors of degree~$3$ supported on the cusps.
\end{proposition}

\begin{proof}
Because $X_\Gamma$ has 9 rational cusps and 3 Galois orbits of cusps with field of definition $\Q(\zeta_3)^+$, there are $\binom{9+3-1}{3}+3=168$ effective divisors of degree~$3$ supported on the cusps.
The nine $\Q$-rational cusps of $X_\Gamma$ lie above three rational points of $X_\Gamma/H$, and also above three rational points of $X_\Gamma/H'$.  Furthermore, none of the three Galois orbits of cusps with field of definition $\Q(\zeta_7)^+$ lies over a single rational point of $X_\Gamma/H$ or $X_\Gamma/H'$.
This implies that the 168 effective divisors of degree~$3$ supported on the cusps form 164 linear equivalence classes, namely 162 consisting of 1 divisor and 2 consisting of 3 divisors.

For each of the 432 points $x\in J_\Gamma(\F_3)$, we compute the least $r\ge0$ such that $x$ is of the form $[D-rO]$ for some effective divisor $D$ of degree~$r$ on $(X_\Gamma)_{\F_3}$.
This yields exactly 164 points in $J_\Gamma(\F_3)$ of the form $[D-3O]$ with $D$ an effective divisor of degree~$3$ on $(X_\Gamma)_{\F_3}$. Therefore at most 164 points in $J_\Gamma(\Q)$ have this property, and since we already have 164 points in $J_\Gamma(\Q)$ that are represented by effective divisors of degree~$3$ supported on the cusps, we are done.
\end{proof}

\begin{proof}[Proof of Theorem~\ref{theorem:cubic}]
An elliptic curve $E$ over a cubic field $K$ with an embedding of $C_2\times C_{14}$ defines an effective divisor $D$ of degree~$3$ on $X_\Gamma$, which we can view as a $\Q$-rational point of $\Sym^3 X_\Gamma$. Then $\phi(D)$ is a $\Q$-rational point of the image of~$\phi$ in~$J_\Gamma$.
By Proposition \ref{prop:1} and the fact that $D$ is evidently not supported on the cusps, $D$ lies in one of the two copies of $\PP^1_\Q$ inside $\Sym^3 X_\Gamma$ that are contracted under $\phi$.
It follows that $D$ is the inverse image of a $\Q$-rational point on one of the two rational curves $X_\Gamma/H$ and $X_\Gamma/H'$ under the maps $q_H$ and $q_{H'}$, respectively.
This implies that $K$ is normal over~$\Q$.
It is known that the field of definition of the two elliptic points of $X_*(7)$ equals $\Q(\zeta_3)$; see for example \cite[\S\,4.4]{elk}.
Thus $D$ lies above a non-elliptic point $s\in X_*(7)(\Q)$, and $E$ is the base change to~$K$ of the fibre at~$s$ of the universal elliptic curve over the complement of the cusps and elliptic points in $X_*(7)$.
We conclude that $E$ is defined over $\Q$.
\end{proof}

\begin{remark}
Given an elliptic curve $E$ over a cubic field $K$ with a subgroup isomorphic to $C_2\times C_{14}$, the proof of Theorem~\ref{theorem:cubic} yields the following procedure to determine a model of $E$ over~$\Q$.
Choose a point $P$ of order~$7$ in $E(K)$, and write down the unique Weierstrass equation for~$E$ such that the points $P$, $2P$ and~$4P$ lie on the line $y=0$ and the points $3P$, $5P$ and~$6P$ lie on the line $y=-x$.  Then this Weierstrass equation has coefficients in~$\Q$.
\label{remark:explicit-cubic}
\end{remark}

\begin{example}
Consider the cubic field $K=\Q(\alpha)$ of discriminant $31^2$, where $\alpha^3-\alpha^2-10\alpha+8=0$.  The elliptic curve $E$ over~$K$ defined by the Weierstrass equation
$$
\begin{aligned}
  y^2 + xy + y &= x^3 - x^2 + (-3737\alpha^2-8584\alpha+9067)x\\
  &\qquad+(203770\alpha^2+468074\alpha-494427)
\end{aligned}
$$
has torsion subgroup isomorphic to $C_2\times C_{14}$, and the point $P=(14\alpha^2 + 32\alpha - 33, 59\alpha^2 + 136\alpha - 144)$ has order~$7$.  After a change of variables to bring $E$ in the form described by Remark~\ref{remark:explicit-cubic} with respect to~$P$, we obtain a Weierstrass equation with coefficients in $\Q$, namely
$$
y^2 + xy = x^3 - \frac{2^2\cdot11}{3^2\cdot31}x^2 + \frac{2^6\cdot7}{3^5\cdot31^2}x + \frac{2^{12}}{3^9\cdot31^3}.
$$
In fact, $E$ is the base change of the elliptic curve over~$\Q$ with Cremona label 1922c1.
\end{example}

\subsection*{Acknowledgements}
We thank the referee for some helpful suggestions.

\end{document}